\documentclass[12pt]{article}
\usepackage{amsmath}
\usepackage{graphicx}
\usepackage{amsfonts}
\usepackage{verbatim}
\usepackage{latexsym}
\usepackage{dsfont}
\usepackage{color}
\usepackage[bf,SL,BF]{subfigure}

\usepackage{a4}
\usepackage{amsmath}
\allowdisplaybreaks[4]
\makeatletter
\@addtoreset{equation}{section}

\begin{document}

\newcommand{\dup}{\mathrm{d}}
\newcommand{\EE}{\mathbb{E}}
\newcommand{\PP}{\mathbb{P}}
\newcommand{\RR}{\mathbb{R}}
\newcommand{\SM}{\mathbb{S}}
\newcommand{\ZZ}{\mathbb{Z}}
\newcommand{\ind}{\mathbf{1}}
\newcommand{\LL}{\mathbb{L}}
\def\F{{\cal F}}
\def\G{{\cal G}}
\def\P{{\cal P}}
\newcommand\eup{\mathrm{e}}

\newtheorem{theorem}{Theorem}[section]
\newtheorem{lemma}[theorem]{Lemma}
\newtheorem{coro}[theorem]{Corollary}
\newtheorem{defn}[theorem]{Definition}
\newtheorem{assp}[theorem]{Assumption}
\newtheorem{cond}[theorem]{Condition}
\newtheorem{expl}[theorem]{Example}
\newtheorem{prop}[theorem]{Proposition}
\newtheorem{rmk}[theorem]{Remark}
\newtheorem{conj}[theorem]{Conjecture}

\newcommand\nat{\mathbb{N}}

\newcommand\tq{{\scriptstyle{3\over 4 }\scriptstyle}}
\newcommand\qua{{\scriptstyle{1\over 4 }\scriptstyle}}
\newcommand\hf{{\textstyle{1\over 2 }\displaystyle}}
\newcommand\hhf{{\scriptstyle{1\over 2 }\scriptstyle}}
\newcommand\hei{\tfrac{1}{8}}

\newcommand{\eproof}{\indent\vrule height6pt width4pt depth1pt\hfil\par\medbreak}

\def\a{\alpha}
\def\e{\varepsilon} \def\z{\zeta} \def\y{\eta} \def\o{\theta}
\def\vo{\vartheta} \def\k{\kappa} \def\l{\lambda} \def\m{\mu} \def\n{\nu}
\def\x{\xi}  \def\r{\rho} \def\s{\sigma}
\def\p{\phi} \def\f{\varphi}   \def\w{\omega}
\def\q{\surd} \def\i{\bot} \def\h{\forall} \def\j{\emptyset}

\def\be{\beta} \def\de{\delta} \def\up{\upsilon} \def\eq{\equiv}
\def\ve{\vee} \def\we{\wedge}

\def\D{\Delta} \def\O{\Theta} \def\L{\Lambda}
\def\X{\Xi} \def\Si{\Sigma} \def\W{\Omega}
\def\M{\partial} \def\N{\nabla} \def\Ex{\exists} \def\K{\times}
\def\V{\bigvee} \def\U{\bigwedge}

\def\1{\oslash} \def\2{\oplus} \def\3{\otimes} \def\4{\ominus}
\def\5{\circ} \def\6{\odot} \def\7{\backslash} \def\8{\infty}
\def\9{\bigcap} \def\0{\bigcup} \def\+{\pm} \def\-{\mp}
\def\la{\langle} \def\ra{\rangle}

\def\proof{\noindent{\it Proof. }}
\def\tl{\tilde}
\def\trace{\hbox{\rm trace}}
\def\diag{\hbox{\rm diag}}
\def\for{\quad\hbox{for }}
\def\refer{\hangindent=0.3in\hangafter=1}

\newcommand\wD{\widehat{\D}}
\newcommand{\ka}{\kappa_{10}}
\title{
\bf Semi-implicit Euler-Maruyama method for non-linear time-changed stochastic differential equations}

\author
{{\bf Chang-Song Deng}
\\
School of Mathematics and Statistics \\Wuhan University, Wuhan 430072, China.
\\
{\bf  Wei Liu\footnote{Corresponding author. Email: weiliu@shnu.edu.cn, lwbvb@hotmail.com}}
\\
Department of Mathematics \\Shanghai Normal University, Shanghai, 200234, China.
}

\date{}

\maketitle

\begin{abstract}
The semi-implicit Euler-Maruyama (EM) method is investigated to approximate a class of time-changed stochastic differential equations, whose drift coefficient can grow super-linearly and diffusion coefficient obeys the global Lipschitz condition. The strong convergence of the semi-implicit EM is proved and the convergence rate is discussed. When the Bernstein
function of the inverse subordinator (time-change)
is regularly varying at zero, we establish the mean square polynomial stability of the underlying equations. In addition, the numerical method is proved to be able to preserve
such an asymptotic property. Numerical simulations are presented to demonstrate the theoretical results.

\end{abstract}

\section{Introduction}

In this paper, we study the numerical approximations to a class
of time-changed stochastic differential equations (SDEs) which
are of the form
\begin{equation*}
\dup X(t) = f(E(t),X(t))\,\dup E(t) + g(E(t),X(t))\,\dup B(E(t)).
\end{equation*}
Here the coefficients $f$ and $g$ satisfy some regularity
conditions (to be specified in Section \ref{sec:MathPre}),
$B(t)$ represents a standard Brownian motion,
and $E(t)$ is an
independent time-change given by
an inverse subordinator. The rigorous mathematical definitions are postponed to Section \ref{sec:MathPre}.
\par
Since it is in general impossible to derive the
explicit solution to such SDEs, numerical approximations become extremely important when one applies them to model uncertain phenomenon in real life. This paper aims to construct a numerical method for these time-changed SDEs. The strong convergence with the convergence rate and the mean square stability of the numerical method are investigated.
\par
To our best knowledge, \cite{JK2016} is the first paper to study the finite time strong convergence of numerical methods for time-changed SDEs by directly discretizing the equations. In \cite{JK2016}, the authors used the duality principle established in \cite{Kobayashi2011} to construct the Euler-Maruyama (EM) method. In a very recent work \cite{JK2019}, the authors studied the EM method for a larger class of time-changed SDEs without the duality principle. However, both of these two works required the coefficients of the time-changed SDEs to satisfy the global Lipschitz condition. This requirement rules out many interesting SDEs like
\begin{equation*}
\dup X(t) = \left(X(t) - X^3(t)\right) \,\dup E(t)
+ X(t) \,\dup B(E(t)),
\end{equation*}
where some cubic term appears in the drift coefficient. Moreover, the EM is proved to be divergent to SDEs with super-linear growing coefficients \cite{HJK2011a}.
\par
To cope with such super-linearity, we propose the semi-implicit EM method to approximate the SDEs driven by time-changed
Brownian motions in this paper. It should be noted that the semi-implicit EM (also called the backward Euler method) have been studied for approximating different types of SDEs driven by
Brownian motions, see \cite{HMY2007a,Hu1996a,LMYY2018,LM2015,MS2013a,NS2014,Sch1997a,WW2010} and the references therein.
\par
Stabilities in different senses for SDEs driven
by time-changed Brownian motion have been discussed in \cite{Wu2016arXiv1}. See \cite{NN2017,NN2018} for related
results when the driven process is
a time-changed L\'evy process. As far as we know,
however, there is no result concerning the stability analysis for numerical methods for time-changed SDEs.
\par
In the three papers mentioned above, the global Lipschitz condition was required for the coefficients of the equations. In this paper, we study the the mean square stability of the underlying time-changed SDEs, where the global Lipschitz condition on the drift coefficients is not required. Then, we investigate the capability of the semi-implicit EM method to reproduce such a property under the similar condition.
\par
The main contributions of this paper are as follows.
\begin{itemize}
\item The semi-implicit EM method is proved to be convergent to a class of time-changed SDEs and the convergence rate is explicitly given.
\item We establish the mean square stability of the underlying time-changed SDEs. In addition, the numerical solution is proved to be able to preserve such a property.
\end{itemize}

The rest of this paper is organized as follows.
Section \ref{sec:MathPre} is devoted to some mathematical preliminaries for the time-changed SDEs to be considered in this paper, and some necessary lemmas. The strong convergence of the numerical method is proved in Subsection \ref{subsec:strongconvergence}, and the mean square stabilities of both underlying and numerical solutions are shown in Subsection \ref{subsec:stability}. In Section \ref{sec:numsim}, we present numerical simulations to demonstrate the
theoretical results derived
in Section \ref{sec:mainrst}.

\section{Preliminaries} \label{sec:MathPre}

Throughout this paper, unless otherwise specified, we will use the following notation.
Let $|\cdot|$ be the Euclidean norm in $\RR^{d}$ and $\la x,y \ra$ be the inner product of vectors $x,y\in \RR^{d}$. If A is a vector or matrix, its transpose is denoted by $A^T$. If A is a matrix, its trace norm is denoted by $|A|=\sqrt{trace(A^TA)}$. For two real numbers $u$ and $v$, we
use $u\we v=\min(u,v)$ and $u\vee v=\max(u,v)$.
\par
Moreover, let $(\Omega, \F, \PP)$ be a complete probability space with a filtration $\left\{\F_t\right\}_{t \ge 0}$ satisfying the usual conditions (that is, it is right continuous and increasing while $\F_0$ contains all $\PP$-null sets). Let $B(t)= (B_1(t), B_2(t), ..., B_m(t))^T$ be an $m$-dimensional
$\F_t$-adapted standard Brownian motion. Let $\EE$ denote the expectation under the probability measure $\PP$.
\par

Let $D(t)$ be an $\F_t$-adapted subordinator (without killing), i.e.\ a nondecreasing L\'{e}vy process on $[0,\infty)$ starting
at $D(0)=0$. The Laplace transform of $D(t)$ is of
the form
$$
    \EE\,\eup^{-rD(t)} = \eup^{-t \phi(r)},\quad r>0,\,t\geq0,
$$
where the characteristic (Laplace)
exponent $\phi:(0,\infty)\rightarrow(0,\infty)$ is a Bernstein
function with $\phi(0+):=\lim_{r\downarrow0}\phi(r)=0$,
i.e.\ a $C^\infty$-function such
that $(-1)^{n-1}\phi^{(n)}\geq0$ for all $n\in\nat$.
Every such $\phi$ has a unique L\'{e}vy--Khintchine representation
$$
    \phi(r)
    =\vartheta r+\int_{(0,\infty)}\left(1-\eup^{-rx}\right) \,\nu(\dup x),\quad r>0,
$$
where $\vartheta\geq0$ is the drift parameter and
$\nu$ is a L\'{e}vy measure on $(0,\infty)$
satisfying $\int_{(0,\infty)}(1\wedge x)
\,\nu(\dup x)<\infty$. We will focus on the case
that $t\mapsto D(t)$ is a.s.\ strictly increasing,
i.e.\ $\vartheta>0$ or $\nu(0,\infty)=\infty$;
obviously, this is also equivalent
to $\phi(\infty):=\lim_{r\rightarrow\infty}\phi(r)=\infty$.

\par
Let $E(t)$ be the (generalized, right-continuous)
inverse of $D(t)$, i.e.
\begin{equation*}
E(t) := \inf\{ s\geq0\,;\,D(s) > t \}, \quad t \geq 0.
\end{equation*}
We call $E(t)$ an inverse subordinator associated with the Bernstein function $\phi$. Note that $t\mapsto E(t)$
is a.s.\ continuous and nondecreasing.
\par
We always assume that $B(t)$ and $D(t)$ are
independent. The process $B(E(t))$ is called
a time-changed Brownian motion, which is trapped whenever $t\mapsto E(t)$ is constant. We remark that the jumps of $t\mapsto D(t)$ correspond to flat pieces
of $t\mapsto E(t)$. Due to these traps, the time-change slows down the original Brownian motion $B(t)$, and $B(E(t))$ is understood
as a subdiffusion in the
literature (cf.\ \cite{MS2004,UHMK2018}).

Consider the following time-changed SDE
\begin{equation}\label{tcSDE}
\dup X(t) = f(E(t),X(t))\,\dup E(t)
+ g(E(t),X(t))\,\dup B(E(t)), \quad t\in [0,T],
\end{equation}
with $\EE |X(0)|^\gamma < 0$ for any $\gamma \in (0,\infty)$, where $f:[0,\infty) \times \RR^d \rightarrow \RR^d $ and $g:[0,\infty) \times \RR^{d} \rightarrow \RR^{d\times m}$ are measurable coefficients. We will need the following assumptions
on the drift and diffusion coefficients.
\begin{assp}\label{assp:one-sidedf}
There exists a constant $K_1 > 0$ such that,
for all $t\geq0$ and $x,y\in\RR^d$,
\begin{equation*}
\left\la x - y, ~ f(t,x) - f(t,y) \right\ra \leq K_1
|x - y|^2.
\end{equation*}
\end{assp}

\begin{assp}\label{assp:fholdert}
There exist constants $K_2 >0$, $a \geq 2$ and $\gamma \in (0,2]$ such that, for all $t,s\geq0$
and $x\in\RR^d$,
\begin{equation*}
\left| f(t,x) - f(s,x) \right|^2 \leq K_2 \left( 1 + |x|^a \right) |t - s|^{\gamma}
\end{equation*}
and
\begin{equation*}
\left| g(t,x) - g(s,x) \right|^2 \leq K_2 \left( 1 + |x|^2 \right) |t - s|^{\gamma}.
\end{equation*}
\end{assp}

\begin{assp}\label{assp:polyf}
Assume that there exist constants $K_3 > 0$ and $b\geq 0$
such that, for all $t\geq0$ and $x,y\in\RR^d$,
\begin{equation*}
\left| f(t,x) - f(t,y) \right|^2 \leq K_3 \left( 1 + |x|^b + |y|^b \right) |x - y|^2
\end{equation*}
and
\begin{equation*}
\left| g(t,x) - g(t,y) \right|^2 \leq K_3  |x - y|^2.
\end{equation*}
\end{assp}

\begin{assp}\label{assp:Kcondfg}
Assume that there exist constant $p \geq 2$ and $K_4 > 0$ such that, for all $t\geq0$ and $x\in\RR^d$,
\begin{equation*}
\la x,~f(t,x) \ra + \frac{p-1}{2} |g(t,x)|^2 \leq K_4 (1 + |x|^2).
\end{equation*}
\end{assp}

By Assumption \ref{assp:polyf}, we can see that there exists a constant $K_5 > 0$ such that
\begin{equation}\label{assp:polyf2}
|f(t,x)|^2 \leq K_5 (1 + |x|^a)
\end{equation}
and
\begin{equation}\label{assp:polyf22}
|g(t,x)|^2 \leq K_5 (1 + |x|^2)
\end{equation}
for all $t\geq0$ and $x\in \RR^d$.

According to the duality principle in \cite{Kobayashi2011}, the time-changed SDE \eqref{tcSDE} and the classical SDE of It\^o type
\begin{equation}\label{SDE}
\dup Y(t) = f(t,Y(t))\,\dup t + g(t,Y(t))\,\dup B(t),~~~Y(0) = X(0),
\end{equation}
have a deep connection. The next lemma states such a relation more precisely, which is borrowed from Theorem 4.2 in \cite{Kobayashi2011}.
\begin{lemma}\label{lemma:equivxandy}
Suppose Assumptions \ref{assp:one-sidedf} to \ref{assp:polyf} hold. If $Y(t)$ is the unique solution to the SDE \eqref{SDE}, then the time-changed process $Y(E(t))$, which is an $\F_{E(t)}$-semimartingale, is the unique solution to the time-changed SDE \eqref{tcSDE}. On the other hand, if $X(t)$ is the unique solution to the time-changed SDE \eqref{tcSDE}, then the process $X(D(t))$, which is an $\F_t$-semimartingale, is the unique solution to the SDE \eqref{SDE}.
\end{lemma}
\par
The plan to numerically approximate the time-changed SDE \eqref{tcSDE} in this paper is as follows. Firstly, we construct the numerical method for the SDE \eqref{SDE}. Secondly, we discretize the inverse subordinator $E(t)$. Then the combination of the numerical solution of the SDE \eqref{SDE} and the discretized inverse subordinator is used to approximate the solution to the time-changed SDE \eqref{tcSDE}.
\par
The semi-implicit EM method for \eqref{SDE} is defined as
\begin{equation}\label{siEM}
y_{i+1} = y_i + f(t_{i+1},y_{i+1})h + g(t_i,y_i)
\Delta B_i,\quad i\in\nat,
\end{equation}
with $y_0 = Y(0)$, where $\Delta B_i$ is the Brownian increment following the normal distribution with the mean 0 and the variance $h>0$ and $t_i = ih$.
\par
Note that under Assumption \ref{assp:one-sidedf}, the semi-implicit EM method \eqref{siEM}
is well defined for any $h \in (0,1/K_1)$ (see for example \cite{MS2013a}). To be more precisely, this means that given $y_i$ is known a unique $y_{i+1}$ can be found. Throughout the paper, we always assume $h \in (0,1/K_1)$.
\par
We also define the piecewise continuous numerical solution by $y(t):= y_i$ for $t \in [t_i,t_{i+1})$, $i\in\nat$.
\par
We follow the idea in \cite{GM2010} to approximate the
inverse subordinator $E(t)$ in a time interval $[0,T]$ for any given $T>0$. Firstly, we simulate the path of $D(t)$ by $D_h(t_i) = D_h(t_{i-1} )+ \Delta_i$ with $D_h(0) = 0$, where $\Delta_i$ is independently identically sequence with $\Delta_i = D(h)$ in distribution. The procedure is stopped when
\begin{equation*}
T \in [ D_h(t_{n}), D_h(t_{n+1})),
\end{equation*}
for some $n$. Then the approximate $E_h(t)$ to $E(t)$ is generated by
\begin{equation}\label{findEht}
E_h(t) = \big(\min\{n; D_h(t_n) > t\} - 1\big)h,
\end{equation}
for $t \in [0,T]$. It is easy to see
\begin{equation*}
E_h(t) = ih,\quad\text{when $t \in   \left[ D_h(t_{i}), D_h(t_{i+1})\right)$}.
\end{equation*}

The next lemma will be used as the approximation error of $E_h(t)$ to $E(t)$, whose proof can be found in \cite{JK2016,Mag2009}.
\begin{lemma}\label{Eterror}
For any $t>0$,
\begin{equation*}
E(t) - h \leq E_h(t) \leq E(t)~~~\text{a.s.}
\end{equation*}
\end{lemma}

The following lemma states that the inverse subordinator $E(t)$ is known to have the finite exponential moment, which was proved in \cite{JK2016,MOW2011}. Here, we give an alternative proof, which can, furthermore, provide an
explicit upper bound.
\begin{lemma}\label{expmonfinite}
For any $\delta>0$, there
exists $C=C(\delta)>0$ such that
\begin{equation*}
\EE\,\eup^{\delta E(t)}\leq \eup^{Ct}
\quad \text{for all $t\geq1$}.
\end{equation*}
\end{lemma}
\begin{proof}
By the definition of $E(t)$, it is clear that
\begin{equation*}
\PP \left( E(t) \leq s \right) = \PP \left( D(s) \geq t \right),~~~t,s \geq 0.
\end{equation*}
Note that
\begin{align*}
\EE\,\eup^{\delta E(t)}&=
\int_0^\infty \PP \left(\eup^{\delta E(t)} > r \right)\,\dup r\\
&=1+\int_1^\infty \PP \left( E(t) >\frac1\delta\log r
\right)\,\dup r\\
&=1+\int_1^\infty \PP \left( D\left(\frac1\delta\log r\right)<t
\right)\,\dup r\\
&=1+\delta\int_0^\infty \PP \left( D(r) < t \right)
\eup^{\delta r}\,\dup r.
\end{align*}
Denote by $\phi^{-1}$ the inverse function
of $\phi$. By the Chebyshev inequality,
\begin{align*}
\PP \left( D(r) < t \right)
&=\PP \left(\eup^{-\phi^{-1}(2\delta) D(r)} > \eup^{-\phi^{-1}(2\delta) t} \right)\\
&\leq \eup^{\phi^{-1}(2\delta) t}\, \EE \,  \eup^{-\phi^{-1}(2\delta) D(r)} \\
&=\eup^{\phi^{-1}(2\delta) t}  \eup^{-r\phi(\phi^{-1}(2\delta))}\\
&=\eup^{\phi^{-1}(2\delta) t-2\delta r}.
\end{align*}
Thus, for all $t > 0$,
$$
    \EE\,\eup^{\delta E(t)}
    \leq1+\delta \eup^{\phi^{-1}(2\delta) t}
    \int_0^\infty \eup^{-\delta r}\,\dup r
    =1+\eup^{\phi^{-1}(2\delta) t},
$$
which immediately implies the assertion. \eproof
\end{proof}

The following result is taken from \cite[Theorem 4.1, p.\ 59]{M2008a}.

\begin{lemma}\label{lemma:momentbdonSDE}
Suppose that Assumptions \ref{assp:one-sidedf} to \ref{assp:Kcondfg} hold. Then the solution to \eqref{SDE} satisfies
\begin{equation*}
\EE |Y(t)|^p \leq 2^{\frac{p-2}{2}} \left(1 +  \EE |Y(0)|^p\right)\eup^{pK_4t}\quad \text{for all $t\geq0$}.
\end{equation*}
\end{lemma}

The next lemma is easy; for the sake of completeness and our readers' convenience, we give a brief proof.

\begin{lemma}\label{lemma:holdery}
Suppose that Assumptions \ref{assp:one-sidedf} to \ref{assp:Kcondfg} hold. Then for any $q\in(1,2p/a]$ and
$t,s\geq0$ with $|t - s|\leq1$,
\begin{equation*}
\EE |Y(t) - Y(s)|^q \leq C |t - s|^{q/2}\eup^{Ct},
\end{equation*}
where $C>0$ is a constant independent of $t$ and $s$.
\end{lemma}

\begin{proof}
For any $0\leq s < t$, we derive from \eqref{SDE} that
\begin{equation*}
Y(t) - Y(s) = \int_s^t f(r,Y(r))\,\dup r + \int_s^t g(r,Y(r))\,\dup B(r).
\end{equation*}
By the elementary inequality
\begin{equation}\label{crinequality}
    \left|\sum_{i=1}^nu_i\right|^q
    \leq n^{q-1}\sum_{i=1}^n|u_i|^q,
    \quad u_i\in\RR^d
\end{equation}
with $n=2$, the H\"older inequality
and \cite[Theorem 7.1, p.\ 39]{M2008a}, we get
\begin{align*}
\EE |Y(t) - Y(s)|^q \leq& |2(t-s)|^{q-1} \EE \int_s^t \left| f(r,Y(r)) \right|^q\,\dup r\\
&+ 2^{q/2 - 1} |q(q-1)|^{q/2} |t - s|^{q/2-1}\EE \int_s^t \left| g(r,Y(r)) \right|^q\,\dup r
\end{align*}
Combining this with \eqref{assp:polyf2},
\eqref{assp:polyf22} and Lemma \ref{lemma:momentbdonSDE},
we obtain
\begin{align*}
\EE |Y(t) - Y(s)|^q &\leq C\left( |t-s|^q + |t - s|^q \eup^{Ct} + |t - s|^{q/2} +  |t - s|^{q/2}\eup^{Ct}\right)\\
&\leq C |t - s|^{q/2}\eup^{Ct},
\end{align*}
where $C>0$ is a generic constant independent of $t$ and $s$ that may change from line to line. This completes
the proof.   \eproof
\end{proof}

\section{Main results} \label{sec:mainrst}

\subsection{Strong convergence} \label{subsec:strongconvergence}
Briefly speaking, the following theorem states the strong convergence with the rate of $(\gamma\wedge1)/2$ of the semi-implicit EM method, which is not surprising. But to our best knowledge, it seems that no existing result fulfills our needs in this paper. In Theorem \ref{thm:seEM}, we need to track the temporal variable $t$ as we will replace it by $E(t)$ in Theorem \ref{thm:timechangenum}. In addition, it seems that no such a result exists on the semi-implicit EM method for non-autonomous SDEs.
\begin{theorem}\label{thm:seEM}
Suppose that Assumptions \ref{assp:one-sidedf} to \ref{assp:Kcondfg} hold with $p \geq 2 (a\vee b)$ and the step size satisfies $h <1/(2(K_1+1))$. Then the semi-implicit EM method \eqref{siEM} is convergent
to \eqref{SDE} with
\begin{equation*}
\EE \left| Y(t) - y(t) \right|^2 \leq C h^{\gamma\wedge1}
\,\eup^{C t},\quad t\geq0,
\end{equation*}
where $C$ is a constant independent of $t$ and $h$.
\end{theorem}
\begin{proof}
From \eqref{SDE} and \eqref{siEM}, it holds that
for $i=1,2,\ldots$,
\begin{align*}
Y(t_{i+1}) - y_{i+1} =& \left( Y(t_i) - y_i \right) + \int_{t_i}^{t_{i+1}} \left( f(s,Y(s)) - f(t_{i+1},y_{i+1}) \right)\,\dup s\\
&+ \int_{t_i}^{t_{i+1}} \left( g(s,Y(s)) - g(t_i,y_i) \right)\,\dup B(s).
\end{align*}
Taking square on both sides yields
\begin{equation*}
\left| Y(t_{i+1}) - y_{i+1} \right|^2 = I_1 + I_2,
\end{equation*}
where
\begin{align*}
I_1&:= \left\la Y(t_{i+1}) - y_{i+1}, ~\int_{t_i}^{t_{i+1}} \left( f(s,Y(s)) - f(t_{i+1},y_{i+1}) \right)
\,\dup s\right\ra \\
&=\int_{t_i}^{t_{i+1}} \left\la Y(t_{i+1}) - y_{i+1}, ~  f(s,Y(s)) - f(t_{i+1},y_{i+1}) \right\ra \,\dup s
\end{align*}
and
\begin{equation*}
I_2 := \left\la Y(t_{i+1}) - y_{i+1}, ~ \left( Y(t_i) - y_i \right) +  \int_{t_i}^{t_{i+1}} \left( g(s,Y(s)) - g(t_i,y_i) \right)\,\dup B(s) \right\ra.
\end{equation*}

To estimate $I_1$, we rewrite the integrand of $I_1$ into three parts
\begin{align*}
&\left\la Y(t_{i+1}) - y_{i+1}, ~  f(s,Y(s)) - f(t_{i+1},y_{i+1}) \right\ra \\
&=\left\la Y(t_{i+1}) - y_{i+1}, ~  f(t_{i+1},Y(t_{i+1})) - f(t_{i+1},y_{i+1}) \right\ra \\
&\quad+ \left\la Y(t_{i+1}) - y_{i+1}, ~  f(s,Y(t_{i+1})) - f(t_{i+1},Y(t_{i+1})) \right\ra \\
&\quad+ \left\la Y(t_{i+1}) - y_{i+1}, ~  f(s,Y(s)) - f(s,Y(t_{i+1})) \right\ra \\ \nonumber
&=:I_{11} + I_{12} + I_{13}.
\end{align*}
Using Assumption \ref{assp:one-sidedf}, we obtain
\begin{equation*}
I_{11} \leq K_1 \left| Y(t_{i+1}) - y_{i+1} \right|^2.
\end{equation*}
Applying the elementary inequality
\begin{equation}\label{hgf65fds}
    \la u,v\ra\leq\frac{|u|^2+|v|^2}{2},
    \quad u,v\in\RR^d,
\end{equation}
we have
\begin{equation*}
I_{12} \leq \frac{1}{2} \left| Y(t_{i+1}) - y_{i+1} \right|^2 + \frac{1}{2}  \left| f(s,Y(t_{i+1})) - f(t_{i+1},Y(t_{i+1})) \right|^2.
\end{equation*}
By Assumption \ref{assp:fholdert}, we can see
\begin{equation*}
\left| f(s,Y(t_{i+1})) - f(t_{i+1},Y(t_{i+1})) \right|^2 \leq K_2 \left( 1 + |Y(t_{i+1})|^a \right) |s - t_{i+1}|^{\gamma}.
\end{equation*}
Thus,
\begin{equation*}
I_{12} \leq \frac{1}{2} \left| Y(t_{i+1}) - y_{i+1} \right|^2 + \frac{K_2}{2} \left( 1 + |Y(t_{i+1})|^a \right) |s - t_{i+1}|^{\gamma}.
\end{equation*}
Applying the elementary inequality \eqref{hgf65fds} and Assumption \ref{assp:polyf} gives
\begin{align*}
I_{13} &\leq \frac{1}{2} \left| Y(t_{i+1}) - y_{i+1} \right|^2 + \frac{1}{2} \left| f(s,Y(s)) - f(s,Y(t_{i+1})) \right|^2 \\
&\leq \frac{1}{2} \left| Y(t_{i+1}) - y_{i+1} \right|^2 + \frac{K_3}{2} \left( 1 + |Y(s)|^b + |Y(t_{i+1})|^b \right) \left|Y(s) - Y(t_{i+1}) \right|^2.
\end{align*}
Combining the upper bound estimates of $I_{11}$, $I_{12}$ and $I_{13}$, we conclude that
\begin{gather}\label{esI_1}
\begin{aligned}
I_1 &\leq \int_{t_i}^{t_{i+1}} \bigg(  (K_1 + 1) \left| Y(t_{i+1}) - y_{i+1} \right|^2 + \frac{K_2}{2} \left( 1 + |Y(t_{i+1})|^a \right) |s - t_{i+1}|^{\gamma}\\
&\quad+ \frac{K_3}{2} \left( 1 + |Y(s)|^b + |Y(t_{i+1})|^b \right) \left|Y(s) - Y(t_{i+1}) \right|^2 \bigg)\,\dup s.
\end{aligned}
\end{gather}
By the H\"older inequality, we find
\begin{align*}
&\EE \bigg( \left( 1 + |Y(s)|^b + |Y(t_{i+1})|^b  \right)\left|Y(s) - Y(t_{i+1}) \right|^2 \bigg)\\
&\qquad\qquad\leq\left(\EE \left( 1 + |Y(s)|^b + |Y(t_{i+1})|^b  \right)^2\right)^{1/2} \left(\EE \left|Y(s) - Y(t_{i+1}) \right|^4 \right)^{1/2}.
\end{align*}
Taking expectations on both sides of \eqref{esI_1} and applying Lemmas \ref{lemma:momentbdonSDE} and \ref{lemma:holdery}, we obtain
\begin{gather}\label{estI1final}
\begin{aligned}
\EE I_1 &\leq (K_1 + 1)h \EE \left| Y(t_{i+1}) - y_{i+1} \right|^2 + C h^{\gamma + 1} + C h^{\gamma + 1}\,\eup^{C t_{i+1}} + C h^2\,\eup^{Ct_{i+1}}\\
&\leq (K_1 + 1) h \EE \left| Y(t_{i+1}) - y_{i+1} \right|^2 + C h^{\gamma + 1}\,\eup^{Ct_{i+1}},
\end{aligned}
\end{gather}
where (and in what follows) $C$ is a generic constant independent of $i$ and the step size $h$
that may change from line to line.
\par
Next, we bound $I_2$. Applying the elementary inequality \eqref{hgf65fds} again, we have
\begin{align*}
I_2 &\leq \frac{1}{2} \left| Y(t_{i+1}) - y_{i+1} \right|^2 + \frac{1}{2} \left|  \left( Y(t_i) - y_i \right) +  \int_{t_i}^{t_{i+1}} \left( g(s,Y(s)) - g(t_i,y_i) \right)\,\dup B(s)  \right|^2 \\
&=:\frac{1}{2} \left| Y(t_{i+1}) - y_{i+1} \right|^2 + \frac{1}{2} I_{21}.
\end{align*}
Taking expectation on both sides and using
the It\^o isometry, it follows that
\begin{equation*}
\EE I_{21} \leq \EE |Y(t_i) - y_i|^2 +  \EE \int_{t_i}^{t_{i+1}} \left| g(s,Y(s)) - g(t_i,y_i) \right|^2\,\dup s.
\end{equation*}
Rewriting the integrand of the second term on the right hand side, and using the elementary inequality \eqref{crinequality}
with $n=3$ and $q=2$ and Assumptions \ref{assp:fholdert} and \ref{assp:polyf}, we can see
\begin{align*}
\left| g(s,Y(s)) - g(t_i,y_i) \right|^2 &\leq 3 \bigg( \left| g(s,Y(s)) - g(s,Y(t_i)) \right|^2 + \left| g(s,Y(t_i)) - g(t_i,Y(t_i)) \right|^2\\
&\quad+ \left| g(t_i,Y(t_i)) - g(t_i,y_i) \right|^2 \bigg) \\
&\leq 3 \left( K_3 |Y(s) - Y(t_i)|^2 + K_2 (1 + |Y(t_i)|^2)|s - t_i|^{\gamma} + K_3 |Y(t_i) - y_i|^2 \right).
\end{align*}
Now applying Lemmas \ref{lemma:momentbdonSDE} and \ref{lemma:holdery} gives
\begin{equation}\label{estI2final}
\EE I_2 \leq \frac{1}{2}\,\EE \left| Y(t_{i+1}) - y_{i+1} \right|^2 + \frac{1+3K_3 h}{2}\,
\left| Y(t_i) - y_i \right|^2 + C h^{(1+\gamma)\wedge2
}\,\eup^{C t_{i+1}}.
\end{equation}
Combining \eqref{estI1final} and \eqref{estI2final} yields
\begin{align*}
\EE \left| Y(t_{i+1}) - y_{i+1} \right|^2 \leq& \left(\frac{1}{2} + h(K_1 + 1)\right) \EE \left| Y(t_{i+1}) - y_{i+1} \right|^2 \\
&+ \frac{1+3K_3 h}{2}\,\EE \left| Y(t_i) - y_i \right|^2 + C h^{(1+\gamma)\wedge2
}\,\eup^{C t_{i+1}},
\end{align*}
which implies that
\begin{equation*}
\EE \left| Y(t_{i+1}) - y_{i+1} \right|^2 \leq \frac{1+3K_3h}
{1 - 2h(K_1 + 1)}\left(\EE \left| Y(t_i) - y_i \right|^2 +
Ch^{(1+\gamma)\wedge2}\,\eup^{C t_{i}}\right).
\end{equation*}
Now summing both sides from $0$ to $i-1$ yields
\begin{equation*}
\sum_{l=1}^i \EE \left| Y(t_{l}) - y_{l} \right|^2 \leq \frac{1+3K_3h}{1 - 2h(K_1 + 1)} \left(\sum_{l=0}^{i-1}\EE \left| Y(t_{l}) - y_{l} \right|^2 + iC h^{(1+\gamma)\wedge2}\,\eup^{C t_{i}}\right).
\end{equation*}
Due to the fact that $ih = t_i \leq \eup^{C t_{i}}$, from combining same terms together on both sides we can derive
\begin{equation*}
\EE \left| Y(t_i) - y_i \right|^2 \leq \frac{h(3K_3 + 2K_1 +2)}{1 - 2h(K_1 + 1)} \sum_{l=0}^{i-1} \EE \left| Y(t_{l}) - y_{l} \right|^2 + C h^{\gamma\wedge1
}\,\eup^{C t_{i}}.
\end{equation*}
By the discrete version of the Gronwall inequality, we have
\begin{equation}\label{ptwscov}
\EE \left| Y(t_i) - y_i \right|^2 \leq C h^{\gamma\wedge1}\,\eup^{C t_{i}}.
\end{equation}
Moveover, when $t \in [t_i,t_{i+1})$
for some $i=1,2,\ldots$, Lemma \ref{lemma:holdery} and \eqref{ptwscov} yield
\begin{align*}
\EE \left| Y(t) - y(t) \right|^2  &=\EE \left| Y(t) - y_i \right|^2\\
&\leq2 \EE \left| Y(t) - Y(t_i) \right|^2 + 2 \EE \left| Y(t_i) - y_i \right|^2 \\
&\leq Ch\,\eup^{Ct}+C h^{\gamma\wedge1}\,\eup^{Ct_i}\\
&\leq C h^{\gamma\wedge1}\,\eup^{C t}.
\end{align*}
Therefore, the proof is completed. \eproof
\end{proof}

\begin{theorem}\label{thm:timechangenum}
Suppose that Assumptions \ref{assp:one-sidedf} to \ref{assp:Kcondfg} hold with $p > 2 (a\vee b)$ and the step size satisfies $h <1/(2(K_1+1))$. Then the combination of the semi-implicit EM solution, $y(t)$, and the discretized inverse subordinator, $E_h(t)$, i.e. $y(E_h(t))$, converges strongly to the solution of \eqref{tcSDE} with
\begin{equation*}
\EE \left| X(T) - y(E_h(T))\right|^2 \leq C h^{\gamma\wedge1}\eup^{C T},
\end{equation*}
where $C$ is a constant independent of $T$ and $h$.
\end{theorem}
\begin{proof}
By Lemma \ref{lemma:equivxandy} and
\eqref{crinequality} with $n=2$ and $q=2$,
\begin{align*}
\EE \left| X(T) - y(E_h(T))\right|^2 &= \EE \left| Y(E(T)) - y(E_h(T))\right|^2 \\
&\leq  2 \EE \left| Y(E(T)) - Y(E_h(T))\right|^2 + 2 \EE \left| Y(E_h(T) - y(E_h(T))\right|^2.
\end{align*}
By Lemmas \ref{Eterror}, \ref{expmonfinite} and \ref{lemma:holdery},  we can see
\begin{equation} \label{yyerror}
\EE \left| Y(E(T)) - Y(E_h(T))\right|^2 \leq C h\,\EE \,\eup^{C E(T)}\leq C h\eup^{C(T\vee1)}.
\end{equation}
On the other hand, it holds from Lemmas \ref{Eterror}
and \ref{expmonfinite} and Theorem \ref{thm:seEM} that
\begin{equation}\label{yYerror}
\EE \left| Y(E_h(T) - y(E_h(T))\right|^2 \leq C h^{\gamma\wedge1}\,\EE\,\eup^{C E_h(T)}
\leq C h^{\gamma\wedge1}\,\EE\,\eup^{C E(T)}
\leq C h^{\gamma\wedge1}\eup^{C(T\vee1)}.
\end{equation}
Combining \eqref{yyerror} and \eqref{yYerror}, we obtain the required assertion.  \eproof
\end{proof}

\subsection{Stability} \label{subsec:stability}

In the section, we always assume the existence and uniqueness of the solutions to \eqref{tcSDE} and \eqref{SDE}. In fact, Assumptions \ref{assp:one-sidedf} to \ref{assp:polyf} are sufficient to guarantee it, but we do not use them explicitly.

A function $F:(0,\infty)\rightarrow(0,\infty)$ is said
to be regularly varying at zero with
index $\alpha\in\RR$ if for any $c>0$,
$$
    \lim_{s\downarrow0}\frac{F(cs)}{F(s)}
    =c^\alpha.
$$
Denote by $\operatorname{RV}_\alpha$ the class of
all regularly varying functions at $0$.
A function $F\in\operatorname{RV}_0$
is said to be slowly varying at $0$. It is clear
that every $F\in\operatorname{RV}_\alpha$
can be rewritten as
$$
    F(s)=s^\alpha\ell(s),
$$
where $\ell$ is a slowly varying function at $0$.

In the following, we will assume that the Bernstein
function $\phi\in\operatorname{RV}_\alpha$ with
$\alpha\in(0,1)$. Typical examples are
\begin{itemize}
  \item
     Let $\phi(r)=r^\alpha\log^\beta(1+r)$
     with $0<\alpha<1$ and $0\leq\beta<1-\alpha$.
     Then $\phi\in\operatorname{RV}_{\alpha+\beta}$;

  \item
     Let $\phi(r)=r^\alpha\log^{-\beta}(1+r)$
     with $0<\beta<\alpha<1$.
     Then $\phi\in\operatorname{RV}_{\alpha-\beta}$;

  \item
     Let $\phi(r)=\log\left(1+r^\alpha\right)$
     with $0<\alpha<1$.
     Then $\phi\in\operatorname{RV}_{\alpha}$;

  \item
     Let $\phi(r)=r^\alpha(1+r)^{-\alpha}$
     with $0<\alpha<1$.
     Then $\phi\in\operatorname{RV}_{\alpha}$.
\end{itemize}
We refer the reader to \cite[Chapter 16]{SSV12} for
more examples of such Bernstein functions.

\begin{lemma}\label{asymptotic}
    If the Bernstein function
    $\phi\in\operatorname{RV}_{\alpha}$ with $\alpha\in(0,1)$, then for any $\lambda>0$
    $$
        \lim_{t\rightarrow\infty}
        \frac{\log\EE\,\eup^{-\lambda E(t)}}{\log t}
        =-\alpha.
    $$
\end{lemma}

\begin{proof}
    Denote by $\mathcal{L}_t[F(t)]$ the Laplace
    transform of a function $F(t)$.
    It follows from \cite[(3.10)]{JK2016} that
    for any $s>0$ and $\lambda>0$,
    $$
        \mathcal{L}_t\left[\EE\,\eup^{-\lambda E(t)}\right](s)
        =\frac{\phi(s)}{s[\phi(s)+\lambda]}.
    $$
    Since $\phi\in\operatorname{RV}_{\alpha}$, we get
    $$
        s\mathcal{L}_t\left[\EE\,\eup^{-\lambda E(t)}\right](s)
        =\frac{\phi(s)}{\phi(s)+\lambda}
        \,\,\sim\,\,\frac{\phi(s)}{\lambda}
        =\frac{1}{\lambda}\,s^\alpha\ell(s),
        \quad s\downarrow0,
    $$
    where $\ell$ is a slowly varying
    function at $0$. Combining this with Karamata's Tauberian
    theorem (cf.\ \cite[Theorem 1.7.6]{BGT87}), it holds that
    \begin{equation}\label{laplaceasy}
        \EE\,\eup^{-\lambda E(t)}\,\,\sim
        \,\,\frac{1}{\lambda\Gamma(1-\alpha)}
        \,t^{-\alpha}\ell\left(\frac{1}{t}\right),
        \quad t\rightarrow\infty.
    \end{equation}
    Noting that $t\mapsto\ell(1/t)$ is slowly
    varying at $\infty$,
    one has (see \cite[Proposition 1.3.6\,(i)]{BGT87})
    $$
        \lim_{t\rightarrow\infty}\frac{\ell(1/t)}{\log t}
        =0,
    $$
    which, together with \eqref{laplaceasy}, implies
    the desired limit.  \eproof
\end{proof}

\begin{theorem}\label{thm:tcSDEstability}
Assume that the Bernstein function $\phi\in\operatorname{RV}_{\alpha}$ with $\alpha\in(0,1)$, and that there exists a
constant $L_1 > 0$ such that
\begin{equation}\label{cond:stab}
\la x,~f(t,x) \ra + \frac{1}{2} |g(t,x)|^2 \leq
-L_1 |x|^2,\quad (x,t) \in \RR^d \times [0,\infty).
\end{equation}
Then
\begin{equation*}
\limsup_{t \rightarrow \infty} \frac{\log\EE |X(t)|^2}
{\log t} \leq -\alpha.
\end{equation*}
In other words, the solution to \eqref{tcSDE} is mean square polynomially stable.
\end{theorem}
\begin{proof}
Given \eqref{cond:stab},
by \cite[Theorem 4.4, p.\ 130]{M2008a} we know
that the solution to \eqref{SDE} is mean
square exponentially stable
\begin{equation*}
\EE |Y(t)|^2 \leq \eup^{-L_1 t}\,\EE |Y(0)|^2.
\end{equation*}
Using Lemma \ref{lemma:equivxandy}, we obtain
\begin{equation*}
\EE |X(t)|^2 = \EE |Y(E(t))|^2 \leq \EE \, \eup^{-L_1 E(t)}
\,\cdot\EE |Y(0)|^2.
\end{equation*}
It remains to apply Lemma \ref{asymptotic} to
complete the proof. \eproof
\end{proof}

\begin{rmk}
It is interesting to observe that the time-changed SDEs \eqref{tcSDE} is polynomially stable while the dual SDEs \eqref{SDE} is stable in the exponential rate. This may be due to the effect of the time-changed Brownian, which slows down the diffusion.
\end{rmk}
Now, we present our result about the stability of the semi-implicit EM method.
\begin{theorem}\label{thm:numstability}
Assume that the Bernstein function $\phi\in\operatorname{RV}_{\alpha}$ with $\alpha\in(0,1)$, and that there exist positive constants $L_2$ and $L_3$ with $2L_2 > L_3$ such that
\begin{equation}\label{cond:numstab}
\la x,~f(t,x) \ra \leq - L_2 |x|^2~~~\text{and}~~~|g(t,x)|^2 \leq L_3 |x|^2, \quad (x,t) \in \RR^d \times [0,\infty).
\end{equation}
Then
\begin{equation*}
\limsup_{t \rightarrow\infty} \frac{\log\EE |Y(E_h(t))|^2 }
{\log t} \leq -\alpha.
\end{equation*}
That is to say, the numerical solution to \eqref{tcSDE} is mean square polynomially stable.
\end{theorem}
\begin{proof}
Assume that \eqref{cond:numstab} holds with $2L_2 > L_3$, the standard approach (see for example \cite{MS2013a}) gives
\begin{equation*}
\EE |Y(t_i)|^2 \leq\eup^{-L_4 t_i}\,\EE|Y(0)|^2,
\end{equation*}
where $L_4 = (2L_2 - L_3)/(1 + 2 L_2)$. Now,
replacing $t_i$ by $E_h(t)$ and using
Lemma \ref{Eterror}, we have
\begin{equation*}
\EE |Y(E_h(t))|^2 \leq \EE|Y(0)|^2\,\cdot
\EE\,\eup^{-L_4 E_h(t)}
\leq \EE|Y(0)|^2\,\cdot \eup^{L_4h}\,\cdot
\EE\,\eup^{-L_4 E(t)}.
\end{equation*}
Now, the application of Lemma \ref{asymptotic} yields the desired assertion. \eproof
\end{proof}
\begin{rmk}
It is not hard to see that \eqref{cond:numstab} together with $2L_2 > L_3$ indicates \eqref{cond:stab} in Theorem \ref{thm:tcSDEstability}. Hence, it can be seen from Theorem \ref{thm:numstability} that the semi-implicit EM method can preserve the mean square polynomial stability of the underlying time-changed SDE.
\end{rmk}

\section{Numerical simulations}\label{sec:numsim}

In this section, we will present two numerical examples. The first example is used to illustrate the strong convergence as well as the convergence rate. The second example demonstrates the mean square stability of the numerical stability. Throughout this section, we focus on the case that $E(t)$ is
an inverse $0.9$-stable
subordinator with Bernstein function $\phi(r)=r^{0.9}$.
\begin{expl}
A one-dimensional nonlinear autonomous time-changed SDE
\begin{equation}\label{expl:strcon}
\dup X(t) = \left(X(t) - X^3(t)\right)\,\dup E(t) + X(t) \,\dup B(E(t)), \quad\text{with $X(0)=1$},
\end{equation}
is considered.
\end{expl}
It is not hard to check that Assumptions \ref{assp:one-sidedf} to \ref{assp:polyf} hold for \eqref{expl:strcon}. Therefore, by Theorem \ref{thm:timechangenum} the numerical solution proposed in this paper is strongly convergent to the underlying solution with the rate of $1/2$.
\par
For a given step size $h$, one path of the numerical solution to \eqref{expl:strcon} is simulated in the following way.
\par
\noindent
{\bf Step 1.} The semi-implicit EM method with the step size $h$ is used to simulate the numerical solution, $y(t) = y_i$, when $t \in [ih, (i+1)h)$ for $i=1,2,3,...$, to the duel SDE
\begin{equation*}
\dup Y(t) = \left(Y(t) - Y^3(t)\right)\,\dup t + Y(t)\,
\dup B(t), \quad\text{with $Y(0)=1$}.
\end{equation*}
\par
\noindent
{\bf Step 2.} One path of the subordinator $D(t)$ is simulated with the same step size $h$. (see for example \cite{JUMthesis2015}).
\par
\noindent
{\bf Step 3.} The $E_h(t)$ is found by using \eqref{findEht}.
\par
\noindent
{\bf Step 4.} The combination, $y(E_h(t))$, is used to approximate \eqref{expl:strcon}.
\par
One path of the $0.9$-stable subordinator $D(t)$ is plotted using $h^{-6}$ in Figure \ref{fig:pathDt}. The corresponding inverse subordinator $E(t)$ is drawn in Figure \ref{fig:pathEt}. One path of the numerical solution to Example \ref{expl:strcon} is displayed in Figure \ref{fig:pathxt}.
\par
\begin{figure}[htbp]
\centering
\subfigure[One path of $D(t)$]
{
  \begin{minipage}{4.5cm}
  \label{fig:pathDt}
  \centering
  \includegraphics[scale=0.36]{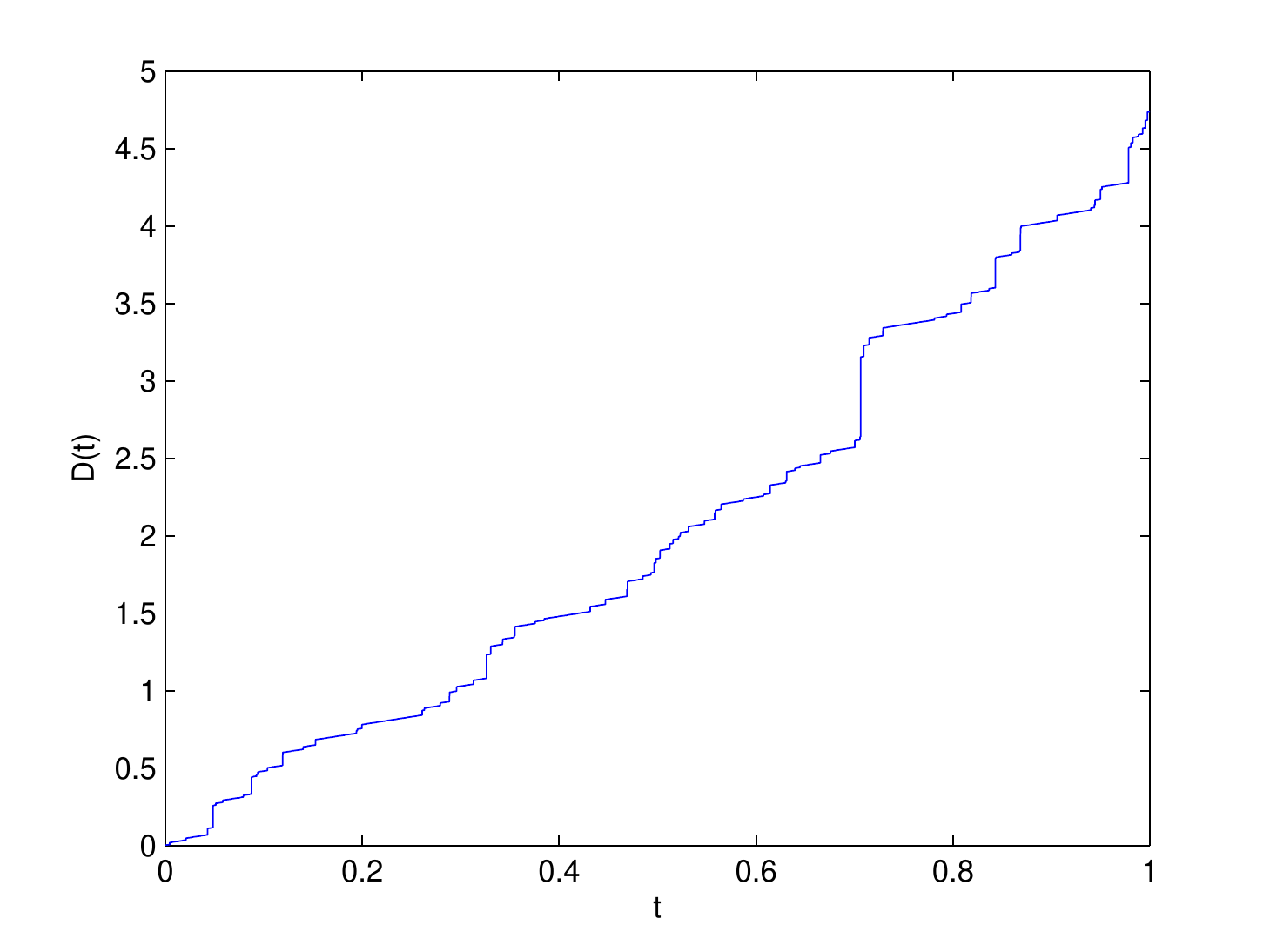}
  \end{minipage}
}
\subfigure[One path of $E(t)$]
{
  \begin{minipage}{4.5cm}
  \label{fig:pathEt}
  \centering
  \includegraphics[scale=0.36]{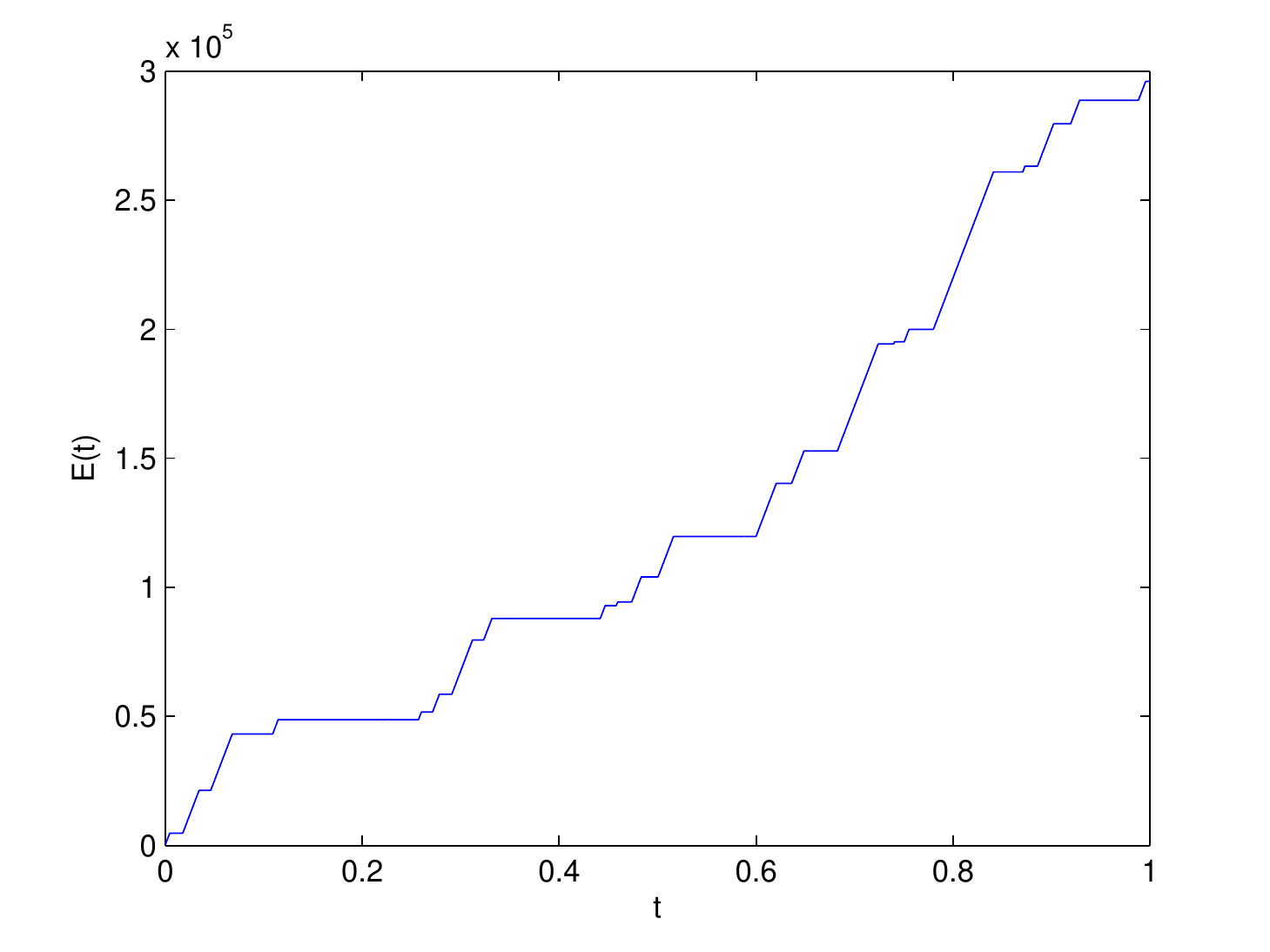}
  \end{minipage}
}
\subfigure[One path of $X(t)$ ]
{
  \begin{minipage}{4.5cm}
  \label{fig:pathxt}
  \centering
  \includegraphics[scale=0.36]{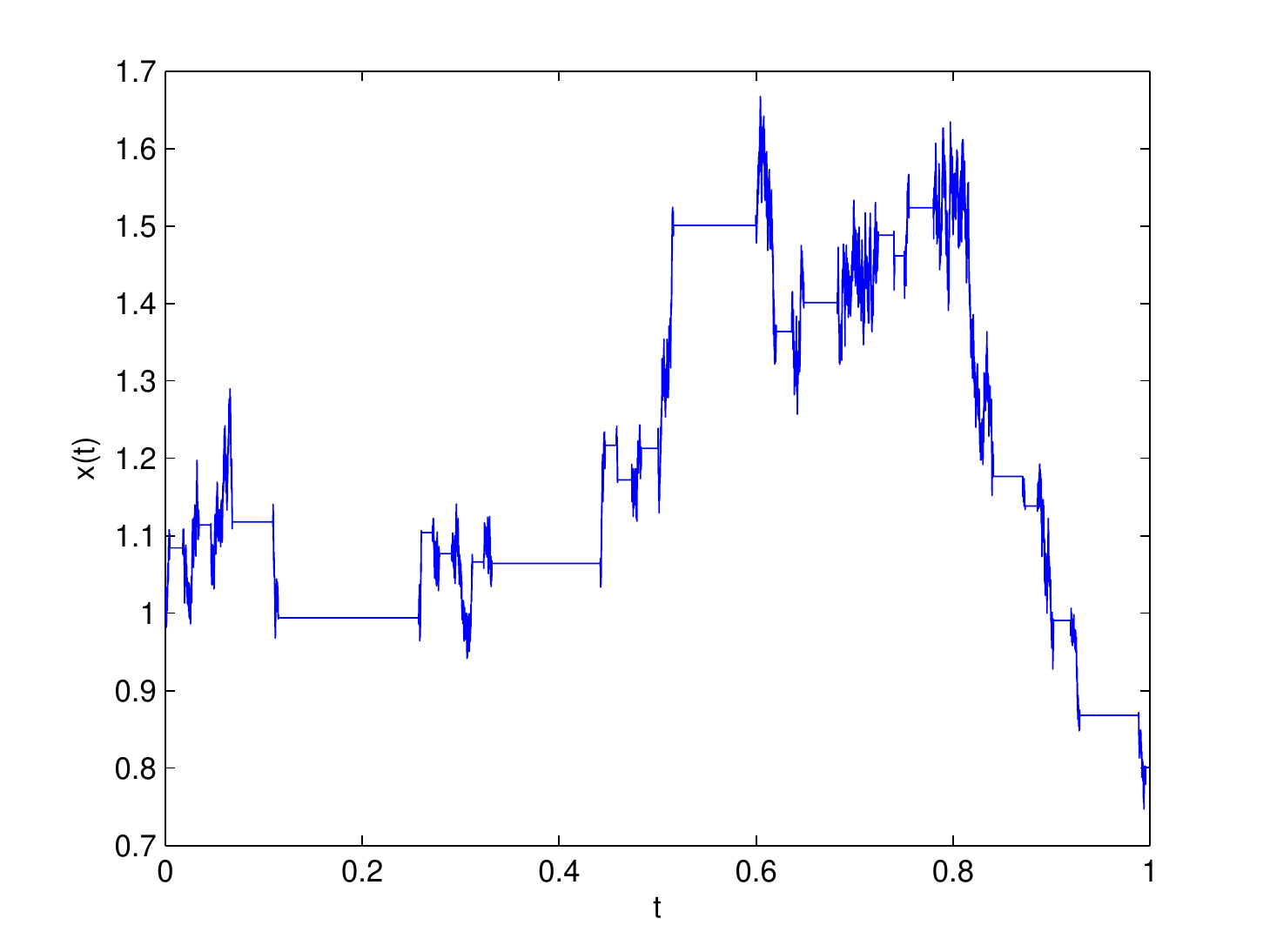}
  \end{minipage}
}
\caption{Numerical simulations of $D(t)$, $E(t)$ and $X(t)$}
\end{figure}

Now we illustrate the strong convergence and the convergence rate. Since the explicit form of the true solution to \eqref{expl:strcon} is hard to obtain. The numerical solution with a small step size, $h_0= 10^{-6}$, is regarded as the true solution. The step sizes of $h=10^{-2}$, $10^{-3}$ and $10^{-4}$ are used to calculated the numerical solutions. For the given step size $h$, the $L^1$ strong error is calculated by
\begin{equation*}
\frac{1}{N} \sum_{i=1}^{N} \left| y_i(E_{h_0}(T)) - y_i(E_h(T)) \right|.
\end{equation*}
\par
Two hundreds ($N=200$) sample paths are used to draw Loglog plot of the $L^1$ error against the step sizes in Figure \ref{Fig:strerr}. The red solid line is the reference line with the slope of 1/2. It can be seen that the strong convergence rate is approximately 1/2. A simple regression also shows that the rate is 0.4996, which is in line with the theoretical one.
\begin{figure}
\centering
  \includegraphics[scale=0.6]{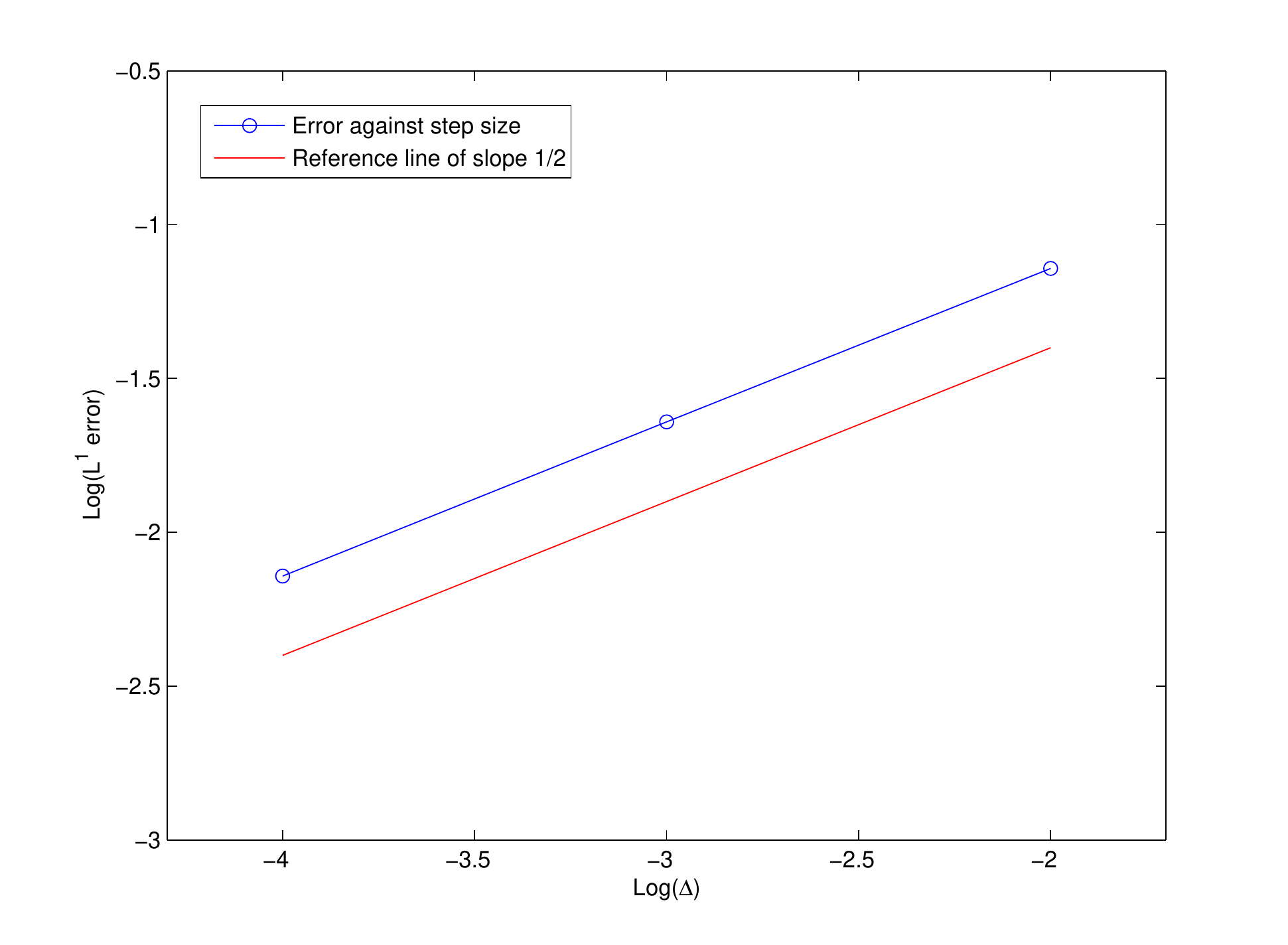}
\caption{The $L^1$ errors between the exact solution and the numerical solutions for step sizes $\Delta=10^{-2},~10^{-3},~10^{-4}$.}
\label{Fig:strerr}
\end{figure}

\begin{expl}
A one-dimensional nonlinear time-changed SDE
\begin{equation}\label{expl:mstab}
\dup X(t) = \left(-X(t) - X^3(t)\right)\,\dup E(t) + X(t) \,\dup B(E(t)), \quad\text{with $X(0)=5$},
\end{equation}
is considered.
\end{expl}
It is not hard to check that \eqref{cond:stab} is satisfied, thus the underlying time-changed SDE is stable in the mean square sense. In addition, \eqref{cond:numstab} holds for \eqref{expl:mstab} indicates the numerical solution is also mean square stable.
\par
One hundred paths are used to draw the mean square of the numerical solutions from $t=0$ to $t=10$. It is clear in Figure \ref{Fig:MSstab} that the second moments of the solution tends to $0$ as the time $t$ advances, which indicates the numerical solution is mean square stable. In addition, five sample paths are displayed in Figure \ref{fig:ASstab}.
\begin{figure}[htbp]
\centering
\subfigure[Mean square of $y(E_h(t))$]
{
  \begin{minipage}{7cm}
  \label{Fig:MSstab}
  \centering
  \includegraphics[scale=0.5]{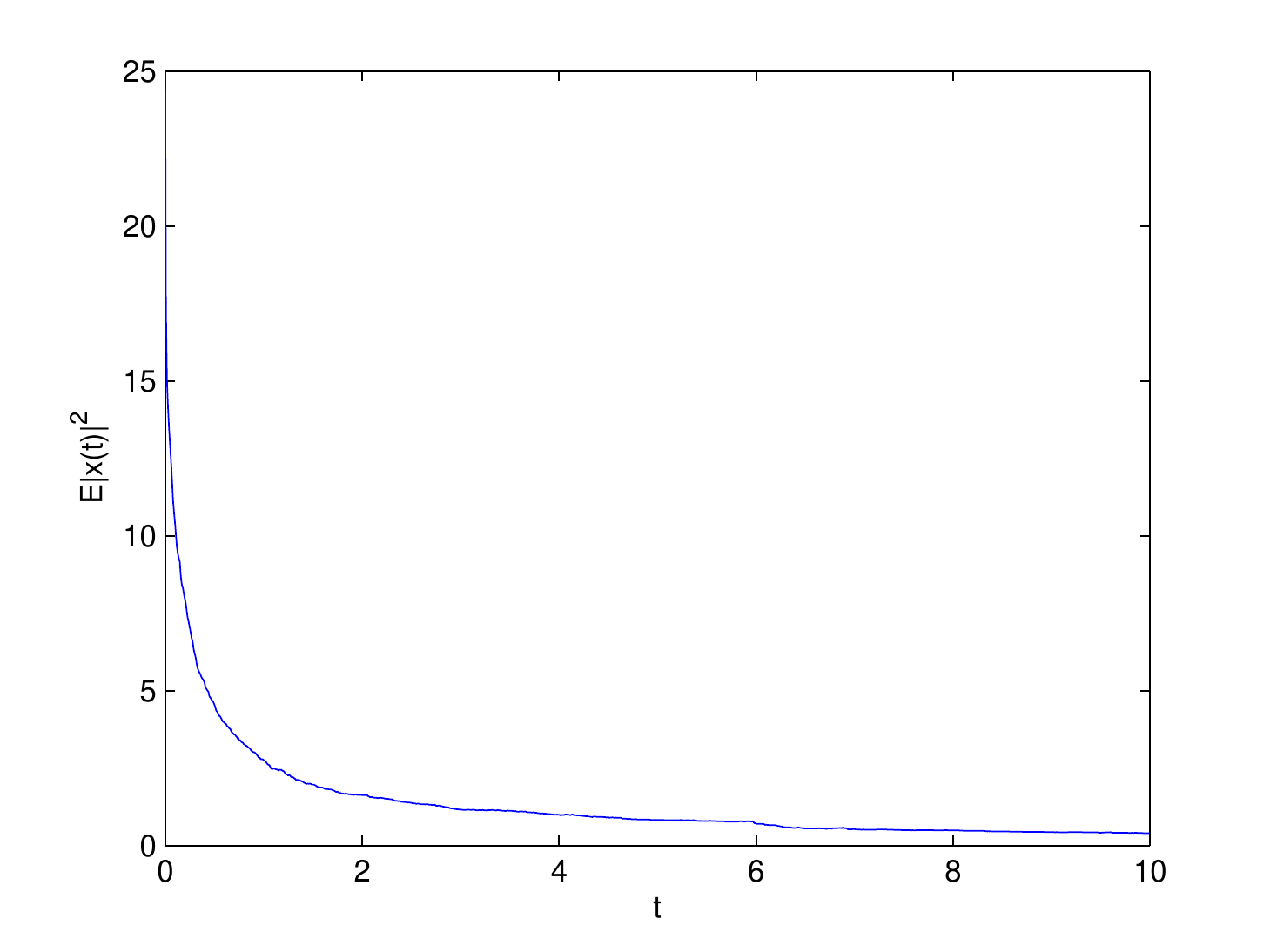}
  \end{minipage}
}
\subfigure[paths of $y(E_h(t))$]
{
  \begin{minipage}{7cm}
  \label{fig:ASstab}
  \centering
  \includegraphics[scale=0.5]{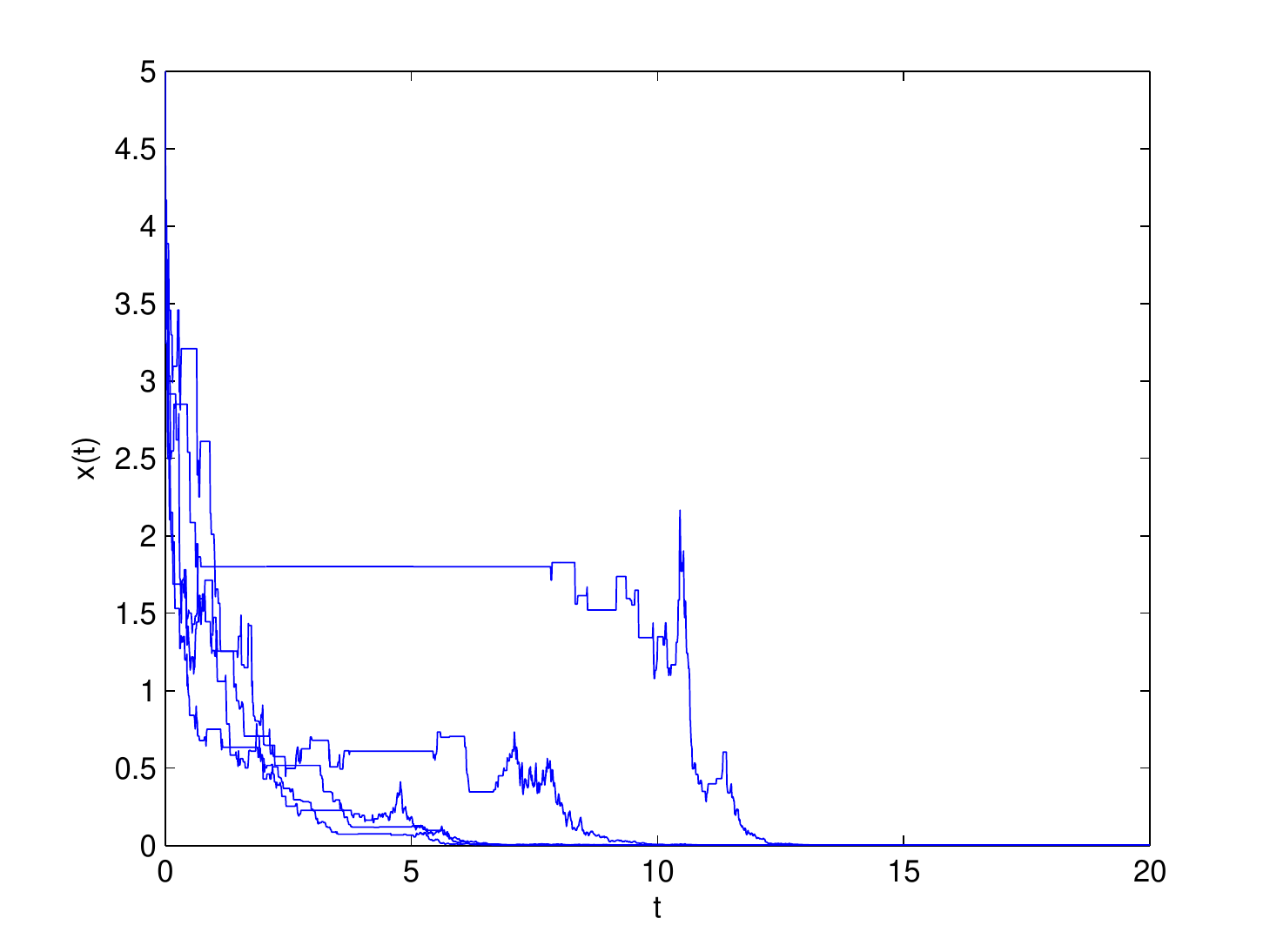}
  \end{minipage}
}
\caption{Stabilities of numerical solutions}
\end{figure}

\section*{Acknowledgement}
Chang-Song Deng would like to thank the National Natural Science Foundation of China (11401442, 11831015).
\par
Wei Liu would like to thank the National Natural Science Foundation of China (11701378, 11871343), ¡°Chenguang Program¡± supported by both Shanghai Education Development Foundation and Shanghai Municipal Education Commission (16CG50), and Shanghai Gaofeng \& Gaoyuan Project for University Academic Program Development, for their financial support.


\begin{thebibliography}{99}

\bibitem{BGT87}
N.~H. Bingham, C.~M. Goldie, and J.~L. Teugels.
\newblock {\em Regular variation}, volume~27 of {\em Encyclopedia of
  Mathematics and its Applications}.
\newblock Cambridge University Press, Cambridge, 1987.

\bibitem{GM2010}
J.~Gajda and M.~Magdziarz.
\newblock Fractional {F}okker-{P}lanck equation with tempered {$\alpha$}-stable
  waiting times: {L}angevin picture and computer simulation.
\newblock {\em Phys. Rev. E (3)}, 82(1):011117, 6, 2010.

\bibitem{HMY2007a}
D.~J. Higham, X.~Mao, and C.~Yuan.
\newblock Almost sure and moment exponential stability in the numerical
  simulation of stochastic differential equations.
\newblock {\em SIAM J. Numer. Anal.}, 45(2):592--609 (electronic), 2007.

\bibitem{Hu1996a}
Y.~Hu.
\newblock Semi-implicit {E}uler-{M}aruyama scheme for stiff stochastic
  equations.
\newblock In {\em Stochastic analysis and related topics, {V} ({S}ilivri,
  1994)}, volume~38 of {\em Progr. Probab.}, pages 183--202. Birkh\"auser
  Boston, Boston, MA, 1996.

\bibitem{HJK2011a}
M.~Hutzenthaler, A.~Jentzen, and P.~E. Kloeden.
\newblock Strong and weak divergence in finite time of {E}uler's method for
  stochastic differential equations with non-globally {L}ipschitz continuous
  coefficients.
\newblock {\em Proc. R. Soc. Lond. Ser. A Math. Phys. Eng. Sci.},
  467(2130):1563--1576, 2011.

\bibitem{JK2019}
S.~Jin and K.~Kobayashi.
\newblock Strong approximation of stochastic differential equations driven by a
  time-changed {B}rownian motion with time-space-dependent coefficients.
\newblock {\em J. Math. Anal. Appl.}, 476(2):619--636, 2019.

\bibitem{JUMthesis2015}
E.~Jum.
\newblock {\em Numerical Approximation of Stochastic Differential Equations
  Driven by L\'evy Motion with Infinitely Many Jumps}.
\newblock PhD thesis, University of Tennessee - Knoxville, 2015.

\bibitem{JK2016}
E.~Jum and K.~Kobayashi.
\newblock A strong and weak approximation scheme for stochastic differential
  equations driven by a time-changed {B}rownian motion.
\newblock {\em Probab. Math. Statist.}, 36(2):201--220, 2016.

\bibitem{Kobayashi2011}
K.~Kobayashi.
\newblock Stochastic calculus for a time-changed semimartingale and the
  associated stochastic differential equations.
\newblock {\em J. Theoret. Probab.}, 24(3):789--820, 2011.

\bibitem{LMYY2018}
X.~Li, Q.~Ma, H.~Yang, and C.~Yuan.
\newblock The numerical invariant measure of stochastic differential equations
  with {M}arkovian switching.
\newblock {\em SIAM J. Numer. Anal.}, 56(3):1435--1455, 2018.

\bibitem{LM2015}
W.~Liu and X.~Mao.
\newblock Numerical stationary distribution and its convergence for nonlinear
  stochastic differential equations.
\newblock {\em J. Comput. Appl. Math.}, 276:16--29, 2015.

\bibitem{Mag2009}
M.~Magdziarz.
\newblock Stochastic representation of subdiffusion processes with
  time-dependent drift.
\newblock {\em Stochastic Process. Appl.}, 119(10):3238--3252, 2009.

\bibitem{MOW2011}
M.~Magdziarz, S.~Orzel, and A.~Weron.
\newblock Option pricing in subdiffusive {B}achelier model.
\newblock {\em J. Stat. Phys.}, 145(1):187--203, 2011.

\bibitem{M2008a}
X.~Mao.
\newblock {\em Stochastic differential equations and applications}.
\newblock Horwood Publishing Limited, Chichester, second edition, 2008.

\bibitem{MS2013a}
X.~Mao and L.~Szpruch.
\newblock Strong convergence rates for backward {E}uler-{M}aruyama method for
  non-linear dissipative-type stochastic differential equations with
  super-linear diffusion coefficients.
\newblock {\em Stochastics}, 85(1):144--171, 2013.

\bibitem{MS2004}
M.~M. Meerschaert and H.-P. Scheffler.
\newblock Limit theorems for continuous-time random walks with infinite mean
  waiting times.
\newblock {\em J. Appl. Probab.}, 41(3):623--638, 2004.

\bibitem{NN2017}
E.~Nane and Y.~Ni.
\newblock Stability of the solution of stochastic differential equation driven
  by time-changed {L}\'{e}vy noise.
\newblock {\em Proc. Amer. Math. Soc.}, 145(7):3085--3104, 2017.

\bibitem{NN2018}
E.~Nane and Y.~Ni.
\newblock Path stability of stochastic differential equations driven by
  time-changed {L}\'{e}vy noises.
\newblock {\em ALEA Lat. Am. J. Probab. Math. Stat.}, 15(1):479--507, 2018.

\bibitem{NS2014}
A.~Neuenkirch and L.~Szpruch.
\newblock First order strong approximations of scalar {SDE}s defined in a
  domain.
\newblock {\em Numer. Math.}, 128(1):103--136, 2014.

\bibitem{SSV12}
R.~L. Schilling, R.~Song, and Z.~Vondra\v{c}ek.
\newblock {\em Bernstein functions. Theory and applications}, volume~37 of {\em
  De Gruyter Studies in Mathematics}.
\newblock Walter de Gruyter \& Co., Berlin, second edition, 2012.

\bibitem{Sch1997a}
H.~Schurz.
\newblock {\em Stability, stationarity, and boundedness of some implicit
  numerical methods for stochastic differential equations and applications}.
\newblock Logos Verlag Berlin, Berlin, 1997.

\bibitem{UHMK2018}
S.~Umarov, M.~Hahn, and K.~Kobayashi.
\newblock {\em Beyond the triangle: {B}rownian motion, {I}to calculus, and
  {F}okker-{P}lanck equation---fractional generalizations}.
\newblock World Scientific Publishing Co. Pte. Ltd., Hackensack, NJ, 2018.

\bibitem{WW2010}
L.~Wang and X.~Wang.
\newblock Convergence of the semi-implicit {E}uler method for stochastic
  age-dependent population equations with {P}oisson jumps.
\newblock {\em Appl. Math. Model.}, 34(8):2034--2043, 2010.

\bibitem{Wu2016arXiv1}
Q.~Wu.
\newblock Stability analysis for a class of nonlinear time-changed systems.
\newblock {\em Cogent Math.}, 3:Art. ID 1228273, 10, 2016.

\end{thebibliography}
\end{document}